\documentclass[12pt, reqno]{amsart}

\usepackage{verbatim}
\usepackage{amssymb, amsmath, amsthm}
\usepackage{hyperref}
\usepackage[alphabetic,lite]{amsrefs}
\usepackage{amscd}   
\usepackage{fullpage}
\usepackage[all]{xy} 

\usepackage{pdflscape}
\DeclareFontEncoding{OT2}{}{} 

\usepackage[usenames, dvipsnames]{color}

\usepackage{graphicx}
\usepackage{tikz}
\usetikzlibrary{matrix,arrows,decorations.pathmorphing}

\usepackage{tikz-cd}

\newtheorem{lemma}{Lemma}[section]
\newtheorem{theorem}[lemma]{Theorem}

\newtheorem{prop}[lemma]{Proposition}
\newtheorem{cor}[lemma]{Corollary}

\newtheorem{claim*}{Claim}
\newtheorem{thm}[lemma]{Theorem}
\newtheorem*{thm*}{Theorem}

\theoremstyle{definition}
\newtheorem{defn}[lemma]{Definition}

\newtheorem{conj}[lemma]{Conjecture}

\newtheorem{remark}[lemma]{Remark}

\newcommand{\Kum}{\textrm{\normalfont Kum}}

\newcommand{\A}{{\bf A}}
\newcommand{\G}{{\bf G}}

\newcommand{\PP}{{\bf P}}

\newcommand{\F}{{\bf F}}
\newcommand{\Q}{{\bf Q}}

\newcommand{\Z}{{\bf Z}}
\newcommand{\NN}{{\bf  N}}

\newcommand{\Xbar}{{\overline{X}}}

\newcommand{\kbar}{{\overline{k}}}

\newcommand{\calO}{{\mathcal O}}

\newcommand{\calV}{{\mathcal V}}

\newcommand{\calX}{{\mathcal X}}

\newcommand{\frakF}{{\mathfrak F}}

\usepackage{mathrsfs} 
\newcommand{\scrA}{{\mathscr A}}

\newcommand{\scrK}{{\mathscr K}}

\DeclareMathOperator{\inv}{inv}

\DeclareMathOperator{\im}{im}

\DeclareMathOperator{\Hom}{Hom}

\DeclareMathOperator{\Aut}{Aut}
\DeclareMathOperator{\Gal}{{Gal}}

\DeclareMathOperator{\Cor}{Cor}
\DeclareMathOperator{\Res}{Res}

\DeclareMathOperator{\Br}{{Br}}

\DeclareMathOperator{\Pic}{Pic}

\DeclareMathOperator{\Spec}{{Spec}}

\DeclareMathOperator{\et}{\textrm{\normalfont \'et}}

\DeclareMathOperator{\N}{N}

\DeclareMathOperator{\GL}{GL}

\DeclareMathOperator{\NS}{NS}

\DeclareMathOperator{\pr}{pr}

\numberwithin{equation}{section}
\numberwithin{table}{section}

\newcommand{\defi}[1]{\textsf{#1}} 

\newcommand{\CH}{\textrm{\normalfont CH}}

\usepackage{scalerel,stackengine}
\stackMath
\newcommand\reallywidehat[1]{%
\savestack{\tmpbox}{\stretchto{%
  \scaleto{%
    \scalerel*[\widthof{\ensuremath{#1}}]{\kern-.6pt\bigwedge\kern-.6pt}%
    {\rule[-\textheight/2]{1ex}{\textheight}}
  }{\textheight}%
}{0.5ex}}%
\stackon[1pt]{#1}{\tmpbox}%
}

\usepackage[margin=2.3cm]{geometry}

\title{Arithmetic of rational points and zero-cycles on Kummer varieties}
\author{Francesca Balestrieri}
\address{Francesca Balestrieri\\ Max-Planck-Institut f{$\ddot{\textrm{u}}$}r Mathematik\\
Vivatsgasse 7\\
 53115 Bonn\\
 Germany.}
\email{fbales@mpim-bonn.mpg.de}

\author{Rachel Newton}
\address{Rachel Newton\\
Department of Mathematics and Statistics\\ 
University of Reading\\
Whiteknights\\
PO Box 220\\
Reading RG6 6AX\\ 
UK.}
   \email{r.d.newton@reading.ac.uk}


\begin{document}
\maketitle

\begin{abstract}
Let $k$ be a number field, let $X$ be a Kummer variety over $k$, and let $\delta$ be an odd integer. In the spirit of a result by Yongqi Liang, we relate the arithmetic of rational points over finite extensions of $k$ to that of zero-cycles over $k$ for $X$.  For example, we show that if the Brauer-Manin obstruction is the only obstruction to the existence of rational points on $X$ over all finite extensions of $k$, then the 2-primary Brauer-Manin obstruction is the only obstruction to the existence of a zero-cycle of degree $\delta$ on $X$ over $k$.
\end{abstract}

\maketitle

\section{Introduction}
Let $X$ be a smooth, proper, geometrically integral variety over a number field $k$ with algebraic closure $\kbar$, 
and let $\Br X:= \mathrm{H}^2_{\et}(X, \G_m)$ be the Brauer group of $X$.  In \cite{Manin}, Manin described a pairing \mbox{\(\langle \ , \  \rangle_{\textrm{BM}} : X(\A_k) \times \Br X \to \Q/\Z\)} such that the set of rational points $X(k)$ is contained in the subset $X(\A_k)^{\Br X}$ of adelic points that are orthogonal to $\Br X$ under the pairing. It can (and sometimes does) happen that $X(\A_k)^{\Br X} = \emptyset $ while $X(\A_k) \neq \emptyset$, yielding an obstruction to the Hasse principle on $X$. This obstruction is known as the Brauer-Manin obstruction (for more details, we refer the reader to e.g. \cite{Skorobogatov-Torsors}). Variations of the Brauer-Manin obstruction can be obtained by considering subsets $B \subset \Br X$ instead of the full Brauer group. For example, one can use the algebraic Brauer group $\Br_1 X:= \ker( \Br X \to \Br \Xbar)$, or the $d$-primary torsion subgroup $\Br X \{d\}$.
Although the Brauer-Manin obstruction may not always suffice to explain the failure of the Hasse principle (see e.g. \cite{Skorobogatov-1999}, \cite{Poonen-2010}), it is nonetheless known to be sufficient for some classes of varieties and is conjectured to be sufficient for geometrically rationally connected varieties (see \cite{CT03}) and K3 surfaces (see \cite{Skorobogatov-K3Conj}). For a survey of recent progress and open problems 
in the field of rational points and obstructions to their existence, we refer the reader to \cite{Wittenberg}.
In general, it is interesting to also consider 0-cycles on varieties, rather than just rational points. Recall that a 0-cycle on $X$ is a formal sum of closed points of $X$, $z=\sum n_i P_i$, with degree $\deg(z)=\sum n_i [k(P_i):k]$. A rational point of $X$ over $k$ is thus a 0-cycle of degree $1$. Conjecturally, the Brauer-Manin obstruction to the existence of 0-cycles of degree $1$ is the only one for smooth, proper, geometrically integral varieties over $k$ (see \cite{KS86}, \cite{CT95}). In this paper,  we focus on the relationship between the arithmetic of the rational points of a variety $X$ over extension fields and the arithmetic of 0-cycles on $X$. We consider the family $\scrK_{k,g}$ of smooth, projective, geometrically integral varieties over $k$ of dimension $g$ which become Kummer varieties upon base change to $\overline{k}$. Our main result (Theorem~\ref{cycKumNU}) implies the following:

\begin{theorem}\label{thm:main_intro}
 Let $X \in \scrK_{k,g}$ and let  $d\in\Z_{>0}$ and $\delta \in \Z$ with $\gcd(d,\delta)=1$. 
 Suppose that, for all finite extensions $l/k$, the $d$-primary Brauer-Manin obstruction is the only obstruction to the existence of rational points over $l$.
 Then the $d$-primary Brauer-Manin obstruction to the Hasse principle is the only obstruction to the existence of 0-cycles of degree $\delta$ on $X$. 
\end{theorem}

 For Kummer varieties, Creutz and Viray (see \cite{CreutzViray-DegBMO}) showed that if the Brauer-Manin obstruction is the only obstruction to the existence of rational points then the $2$-primary Brauer-Manin obstruction is the only obstruction to the existence of rational points. This gives the following corollary in the case where $X$ is a Kummer variety over $k$ (cf. Corollary~\ref{cor:Kummer}).

\begin{cor}\label{cor:Kummer2-primary}
Let $X$ be a Kummer variety over $k$ and let $\delta$ be an odd integer. Suppose that, for all finite extensions $l/k$, the Brauer-Manin obstruction is the only obstruction to the existence of rational points over $l$.
 Then the $2$-primary Brauer-Manin obstruction to the Hasse principle is the only obstruction to the existence of 0-cycles of degree $\delta$ on $X$. 
\end{cor}

In \cite{Skorobogatov-K3Conj}, Skorobogatov conjectured that the Brauer-Manin obstruction is the only obstruction to the existence of rational points on K3 surfaces over number fields. This leads to the following conditional result for Kummer surfaces (cf. Corollary \ref{corsimple}).

\begin{cor}\label{cor:conditional}
Let $\delta$ be an odd integer. Then
Skorobogatov's conjecture that the Brauer-Manin obstruction is the only obstruction to the existence of rational points on K3 surfaces over number fields implies that the 2-primary Brauer-Manin obstruction to the Hasse principle is the only one for 0-cycles of degree $\delta$ on Kummer surfaces over number fields. 
\end{cor}

\begin{remark}Note that we use the phrase `Kummer variety' in a slightly more general way than it is used in some texts. See Section~\ref{sec:Kummer} for the definition of a Kummer variety; a Kummer surface is a Kummer variety of dimension $2$.
\end{remark}

This paper was inspired by Liang's work in \cite[Thm 3.2.1]{Liang}, where he proved for geometrically rationally connected varieties $X$ that if the Brauer-Manin obstruction is the only obstruction to the existence of rational points on $X$ over every finite extension of $k$, then the Brauer-Manin obstruction is the only obstruction to the existence of a 0-cycle of degree 1 on $X$ over $k$ (see also \cite[Prop. 3.4.1]{Liang} for a result about 0-cycles of any degree). He also proved the analogue for weak approximation, in the sense defined in Definition~\ref{defn:BrWA}. Our adaptation of Liang's strategy to the case of (geometric) Kummer varieties also yields the analogous result for this notion of weak approximation.

\begin{theorem}\label{thm:main_intro_WA}
 Let $X \in \scrK_{k,g}$ and let $d\in\Z_{>0}$ and $\delta \in \Z$ with $\gcd(d,\delta)=1$. 
 Suppose that for all finite extensions $l/k$, the $d$-primary Brauer-Manin obstruction is the only obstruction to weak approximation for rational points over $l$.
 Then the $d$-primary Brauer-Manin obstruction to the Hasse principle is the only obstruction to weak approximation for 0-cycles of degree $\delta$ on $X$. 
\end{theorem}

Our key observation, which allows us to employ Liang's method, is that, while it may not be possible to control the growth of the whole of $\Br X/\Br_0 X$ under an extension of the base field, in the case of a variety $X\in\scrK_{k,g}$ the $d$-primary part of $\Br X/\Br_0 X$ is unchanged so long as the base field extension has degree coprime to $d$ and is linearly disjoint from a certain field extension depending on $X$. 
Controlling the growth of the algebraic part $\Br_1 X/ \Br_0 X$ of the Brauer group is easily achieved. We remark that, since the geometric Brauer group of geometrically rationally connected varieties is trivial, the algebraic part of the Brauer group is all that has to be considered in Liang's paper. In the case of Kummer varieties, however, we must also consider the transcendental part $\Br X/ \Br_1 X$ of the Brauer group. This is much harder to deal with than the algebraic part.  In order to control the growth of the transcendental part under extensions of the base field, we use the close relationship between the Brauer group of a Kummer variety and the Brauer group of the underlying abelian variety. The exploitation of the rich structure of abelian varieties is a common theme underlying much of the recent rapid progress in the study of rational points on the related K3 surfaces and Kummer varieties (see e.g. \cite{IeronymouSkoroZarhin}, \cite{SZ-2012},  \cite{IeronymouSkoro},  \cite{Newton15},   \cite{HarpazSkoro}, \cite{VAV16},  \cite{Harpaz}, \cite{CreutzViray-DegBMO}, \cite{SZ-KumVar}). Our work in Section \ref{sec:ControlAbVar} on the transcendental part of the Brauer group of geometrically abelian varieties and geometrically Kummer varieties may be of independent interest.

\begin{remark}After obtaining the results in the present paper, we were made aware of a recent preprint of Ieronymou (\cite{Ieronymou}) in which he uses Liang's approach together with work of Orr and Skorobogatov (\cite{OrrSkoro}) to prove analogues of Liang's results in \cite{Liang} for K3 surfaces. Since Kummer surfaces are both K3 surfaces and Kummer varieties of dimension $2$, there is an overlap between our work and that of Ieronymou (note that Ieronymou's results do not require that the degree of the 0-cycle is odd). We remark, however, that our results are not implied by those of Ieronymou since we work with Kummer varieties of arbitrary dimension.
\end{remark}

Our final result concerns lifting sufficiency of the Brauer-Manin obstruction to the existence of rational points from the base field to uncountably many finite extensions. 

\begin{theorem}\label{thm:liftsuff}
Let $X$ be a Kummer variety  over $k$. Suppose that  $(\Br X/\Br_0 X)\{2\} \cong\Z/2\Z$ and that $X(\A_k) \neq \emptyset$. Then there exist uncountably many finite extensions $l/k$  such that 
\[ X(\A_k)^{\Br X} = \emptyset \Longrightarrow  X(\A_l)^{\Br X_l} = \emptyset.\]
In particular, if the Brauer-Manin obstruction to the Hasse principle  is the only one for $k$-rational points, then it is also the only one  for $l$-rational points, for all field extensions $l/k$ as above. 
\end{theorem}

\subsection{Structure of the paper} In Section~\ref{sec:Kummer}, we recall the definition of a Kummer variety. In Section~\ref{sec:ControlAbVar}, we show that the $d$-primary part of the transcendental Brauer group of a geometrically abelian or Kummer variety remains unchanged when the base field is extended, given certain conditions on the extension. In Section~\ref{sec:ControlKummer}, we use the results of Section~\ref{sec:ControlAbVar} to show that for $X\in\scrK_{k,g}$, the $d$-primary part of the quotient $\Br X/\Br_0 X$ remains unchanged when the base field is extended, given certain conditions on the extension. In Section~\ref{sec:0-cyc}, we prove our main results which are Theorem~\ref{cycKumNU} and its uniform analogue Theorem~\ref{cycKum}, along with Corollaries~\ref{cor:Kummer2-primary} and \ref{cor:conditional}. Finally, in Section~\ref{sec:transfer}, we prove Theorem~\ref{thm:liftsuff}.

\subsection{Notation and terminology}  Throughout this paper, we use the following notation.  In general, given a field $k$ of characteristic 0, we will implicitly assume that we have fixed an algebraic closure $\kbar$ of $k$ once and for all. If $k$ is a number field, we denote by $\Omega_k$ the set of all places of $k$. For a variety $X$ over a field $k$ of characteristic 0 and for $l/k$ a field extension, we let $X_l:= X \times_{\Spec k} \Spec l$ denote the base-change of $X$ to $\Spec l$; we write $\Xbar$ for $X_\kbar$. For a field $k$ of characteristic 0, we denote by $\Gamma_k$ the absolute Galois group $\Gal(\kbar/k)$. Let $d \in \Z_{>0}$. If $X$ is a variety over a field $k$ and $\Br X$ is its Brauer group, we denote by $\Br X [d]$ the $d$-torsion subgroup of the Brauer group and we denote by $\Br X \{d\}:=\varinjlim \Br X [d^n]$ the $d$-primary part  of the Brauer group. For a variety $X$ over a field $k$ of characteristic 0 with structure morphism $s: X \to \Spec k$, the algebraic Brauer group is defined as $\Br_1 X := \ker(\Br X \to \Br \Xbar)$, where the map is the natural one induced by the inclusion $k \subset \kbar$, while the constant Brauer group is defined as $\Br_0 X := \im (\Br k \to \Br X)$, where the map is the natural one induced by $s$. We will refer to the quotient $\Br_1 X/\Br_0 X$ as the \emph{algebraic} part of the Brauer group, and to the quotient $\Br X/ \Br_1 X$ as the \emph{transcendental} part of the Brauer group.\\

\noindent\emph{Acknowledgements.} This collaboration began at the workshop \emph{Rational Points 2017}. The authors are grateful to Michael Stoll for organising such a stimulating workshop.
FB thanks the Max-Planck-Institut f{$\ddot{\textrm{u}}$}r Mathematik for the financial support and for providing excellent working conditions.  The authors are grateful to Daniel Loughran for pointing out the recent preprint \cite{Ieronymou} and thank Evis Ieronymou for his interest in this work.

\section{Kummer varieties}\label{sec:Kummer}

Let $k$ be a field of characteristic 0. We define here our main objects of study.

\begin{defn}[({\cite[\S5.2]{CreutzViray-DegBMO}})] Let $A$ be an abelian variety over $k$. Any $k$-torsor $T$ under $A[2$]
gives rise to a $2$-covering $\rho: V \to A$, where $V$ is the quotient of $A\times_k T$ by the diagonal
action of $A[2]$ and $\rho$ comes from the projection onto the first factor. Then $T =\rho^{-1}(O_A)$ and $V$ has
the structure of a $k$-torsor under $A$. The class of $T$ maps to the class of $V$ under the
map $\mathrm{H}^1_{\et}(k, A[2]) \to \mathrm{H}^1_{\et}(k, A)$ induced by the inclusion of group schemes $A[2] ֒\to A$ and, in
particular, the period of $V$ divides $2$.
Let $\sigma : \tilde{V} \to V$ be the blow-up of $V$ at $T \subset V$. The involution $[-1]: A \to A$ fixes $A[2]$
and induces involutions $\iota$ on $V$ and $\tilde{\iota}$ on $\tilde{V}$ whose fixed point sets are $T$ and the exceptional
divisor, respectively. The quotient $\tilde{V}/\tilde{\iota}$ is geometrically isomorphic to the Kummer variety
$\Kum A$, so in particular it is smooth. We call $\Kum V  :=\tilde{V}/\tilde{\iota}$ the \defi{Kummer variety associated
to $V$ (or $T$)}. A Kummer variety of dimension $2$ is called a Kummer surface.
\end{defn}

 \begin{defn}  Denote by $\scrK_{k,g}$ the family of smooth, projective, geometrically integral varieties $X$ over $k$ of dimension $g$ such that $\overline{X}$ is a Kummer variety over $\overline{k}$.
 \end{defn}

\begin{remark} Let $\scrK^{K3}_{k, 2}$ 
be the family of smooth, projective K3 surfaces $X$ over $k$ such that $ \NS \Xbar$  contains a primitive sublattice isomorphic to the Kummer lattice. Then, by \cite[Theorem~3]{Nikulin}, $\scrK^{K3}_{k, 2}\subset \scrK_{k,2}$. Moreover,
by \cite[Prop. 2.1]{VAV16}, there exists some absolute $M \in \Z_{\geq 1}$ (not depending on $k$) such that for each $X\in \scrK^{K3}_{k, 2}$, there is an extension $k_0/k$ of degree at most $M$ for which $X_{k_0}$ is a Kummer surface over $k_0$. In fact, the proof of  \cite[Prop. 2.1]{VAV16} shows that one can take $M= 2 \cdot \#\GL_{20}(\mathbb{F}_3)$.
\end{remark}

We will use properties of the underlying abelian varieties to deduce results for varieties in  $\scrK_{k,g}$. Therefore, it is convenient to define the following family.

 \begin{defn}
 Denote by $\scrA_{k,g}$ the family of smooth, projective, geometrically integral varieties $X$ over $k$ of dimension $g$ such that $\overline{X}$ is an abelian variety over $\overline{k}$.
 \end{defn}

We write $X\in \scrA_{k,g}\cup \scrK_{k,g}$ to indicate that $X/k$ is a smooth, projective, geometrically integral variety that is either geometrically abelian or geometrically Kummer.

\section{Controlling the transcendental Brauer group of geometrically abelian varieties and geometrically Kummer varieties over field extensions}\label{sec:ControlAbVar}

Let $k$ be a field of characteristic 0. In this section, we show that the $d$-primary part of the transcendental Brauer group of a geometrically abelian or Kummer variety does not grow when the base field is extended, so long as the field extension satisfies certain conditions. 

\begin{prop}\label{prop:nogrowthabvar}
Let $d,n\in \Z_{>0}$. Let $k$ be a field of characteristic 0.
Let $X$ be a smooth, proper, geometrically integral variety over $k$. Let $l/k$ be a finite extension with degree coprime to $d$ such that $\Br\overline{X}[d^n]^{\Gamma_l}=\Br\overline{X}[d^n]^{\Gamma_k}$. 
Then
 the restriction map gives a canonical isomorphism of abelian groups
\[ \Res_{l/k}:(\Br X)[d^n]/(\Br_1 X)[d^n]\xrightarrow{\sim} (\Br X_l)[d^n]/(\Br_1 X_l)[d^n].\]
\end{prop}

\begin{proof}
\noindent\emph{Step 1:} We will show that the corestriction map induces an injection 
\[\Cor_{l/k} : (\Br X_l)[d^n]/(\Br_1 X_l)[d^n] \hookrightarrow (\Br X)[d^n]/(\Br_1 X)[d^n].\]
Let $B\in (\Br X_l)[d^n]$ and suppose that $\Cor_{l/k}B\in \Br_1 X$. By \cite[Lemme 1.4]{CTSk-2013}, 
\begin{equation}\label{eq:CorB}
\Res_{\overline{k}/k}\circ\Cor_{l/k}B=\sum_{\sigma\in\Gamma_k/\Gamma_l}\sigma(\Res_{\overline{k}/l}B),
\end{equation}
where the sum is over a set of coset representatives of $\Gamma_l$ in $\Gamma_k$. Since $\Cor_{l/k}B\in \Br_1 X$, the left-hand side equals zero. Moreover, $\Res_{\overline{k}/l}B\in\Br\overline{X}[d^n]^{\Gamma_l}$. By our hypothesis, this implies that $\Res_{\overline{k}/l}B\in\Br\overline{X}[d^n]^{\Gamma_k}$. Therefore, \eqref{eq:CorB} becomes
\begin{equation}\label{eq:BinBr1}
0=[l:k]\Res_{\overline{k}/l}B.
\end{equation}
Now recall that $[l:k]$ is coprime to $d$. Therefore, \eqref{eq:BinBr1} shows that $\Res_{\overline{k}/l}B=0$, since $B\in(\Br X_l)[d^n]$. Hence, $B\in(\Br_1 X_l)[d^n]$, as required.\\

\noindent\emph{Step 2:} Since $\Cor_{l/k}\circ\Res_{l/k}=[l:k]$ and $[l:k]$ is coprime to $d$, the corestriction map \[\Cor_{l/k}:(\Br X_l)[d^n]/(\Br_1 X_l)[d^n]\to (\Br X)[d^n]/(\Br_1 X)[d^n]\]
is surjective. Step 1 shows that it is an isomorphism and consequently the restriction map
\[\Res_{l/k}:(\Br X)[d^n]/(\Br_1 X)[d^n]  \to  (\Br X_l)[d^n]/(\Br_1 X_l)[d^n]\]
  is also an isomorphism.
\end{proof}

\begin{lemma}\label{lem:ppower}
Let $d \in \Z_{>0}$. Let $k$ be a field of characteristic 0.
Let $X\in\scrA_{k,g}\cup \scrK_{k,g}$. 
Let $l/k$ be a finite extension of degree coprime to $d$ such that the image of $\Gamma_l$ in $\Aut(\Br\overline{X}[d])$ is the same as that of $\Gamma_k$. 
Then for all $n\in\NN$, the image of $\Gamma_l$ in $\Aut(\Br\overline{X}[d^n])$ is the same as that of $\Gamma_k$.
\end{lemma}

\begin{proof}
Factorise $d$ into primes, writing $d=p_1^{e_1}\dots p_r^{e_r}$. Then \[\Br\overline{X}[d^n]=\Br\overline{X}[p_1^{e_1n}]\times \dots \times \Br\overline{X}[p_r^{e_rn}]\] and \[\Aut(\Br\overline{X}[d^n])=\Aut(\Br\overline{X}[p_1^{e_1n}])\times \dots \times \Aut(\Br\overline{X}[p_r^{e_rn}]).\]
Thus, it suffices to prove the lemma when $d$ is a prime $p$. 

\noindent\emph{Step 1:} First we show that there exists $m\in \N$ such that for all $n\in \N$, $\Br\overline{X}[p^n]\cong (\Z/p^n\Z)^{m}$ as abelian groups. We begin by proving this in the case where $\overline{X}=\overline{A}$ is an abelian variety.

The Kummer exact sequence for $\overline{A}$ gives the following exact sequence of Galois modules:
\[0\to \NS\overline{A}/p^n\to \mathrm{H}^2_{\textrm{\'{e}t}}(\overline{A},\mu_{p^n})\to \Br\overline{A}[p^n]\to 0.\]

As an abelian group with no Galois action, \[\mathrm{H}^2_{\textrm{\'{e}t}}(\overline{A},\mu_{p^n})=\mathrm{H}^2_{\textrm{\'{e}t}}(\overline{A},\Z/p^n\Z)=\wedge^2\mathrm{H}^1_{\textrm{\'{e}t}}(\overline{A},\Z/p^n\Z).\]
It is well known that $\mathrm{H}^1_{\textrm{\'{e}t}}(\overline{A},\Z/p^n\Z)$ is isomorphic to $\Hom(A_{p^n},\Z/p^n\Z)$, see \cite[Lemma 2.1]{SZ-2012}, for example. As an abelian group, $A_{p^n}=(\Z/p^n\Z)^{2g}$ and hence \[\wedge^2\Hom(A_{p^n},\Z/p^n\Z)=\wedge^2(\Z/p^n\Z)^{2g}=(\Z/p^n\Z)^{g(2g-1)}.\]
Furthermore, $\NS\overline{A}$ is finitely generated and torsion-free of rank $\rho$ with $1\leq \rho\leq g^2$. Therefore, as an abelian group, $\Br\overline{A}[p^n]\cong (\Z/p^n\Z)^{m}$, where $m=g(2g-1)-\rho$.

Now if $X\in \scrK_{k,g}$ then $\overline{X}=\Kum \overline{A}$ for some abelian variety $A$. Hence, by \cite[Proposition~2.7]{SZ-KumVar}, $\Br\overline{X}=\Br\overline{A}$. This completes Step 1.
  
\noindent\emph{Step 2:}  Since $\Br\overline{X}[p^n]\cong (\Z/p^n\Z)^{m}$ as abelian groups, we have $ \Aut(\Br\overline{X}[p^{n}])\cong\GL_m(\Z/p^n\Z)$.
The inclusion $\Br\overline{X}[p]\subset\Br\overline{X}[p^{n}]$ induces a surjective map
\[\alpha: \Aut(\Br\overline{X}[p^{n}])\to\Aut(\Br\overline{X}[p]),\] which can be identified with the map $\GL_m(\Z/p^n\Z)\to \GL_m(\Z/p\Z)$ given by reduction modulo $p$.  
Therefore, 
$\ker\alpha$ is a $p$-group.

Let $\phi:\Gamma_k\to \Aut(\Br\overline{X}[p^{n}])$ be the map giving the Galois action. The First Isomorphism Theorem gives
\[\#\phi(\Gamma_l)=\#(\phi(\Gamma_l)\cap\ker\alpha) \cdot \#\alpha(\phi(\Gamma_l))\] and \[\#\phi(\Gamma_k)=\#(\phi(\Gamma_k)\cap\ker\alpha)\cdot \#\alpha(\phi(\Gamma_k)).\]
 Note that $\alpha\circ\phi :\Gamma_k\to \Aut(\Br\overline{X}[p])$ is the map giving the Galois action. Hence, $\alpha(\phi(\Gamma_k))=\alpha(\phi(\Gamma_l))$ by assumption. Therefore, 
\[[\phi(\Gamma_k):\phi(\Gamma_l)]=\frac{\#(\phi(\Gamma_k))}{\#(\phi(\Gamma_l))}=\frac{\#(\phi(\Gamma_k)\cap\ker\alpha)}{\#(\phi(\Gamma_l)\cap\ker\alpha)}.\] This is a power of $p$, since $\ker\alpha$ is a $p$-group. On the other hand, $[\phi(\Gamma_k):\phi(\Gamma_l)]$ divides $[\Gamma_k:\Gamma_l]=[l:k]$, which is coprime to $p$ by assumption.
Therefore, $[\phi(\Gamma_k):\phi(\Gamma_l)]=1$, as required.
\end{proof}

\begin{cor}\label{kum-ext3NU} 
Let $d\in \Z_{>0}$.
Let $k$ be a field of characteristic 0.
Let $X\in\scrA_{k,g}\cup \scrK_{k,g}$. Let $F/k$ be a finite extension such that $\Gamma_F$ acts trivially on $\Br\overline{X}[d]$. Let $l/k$ be a finite extension of degree coprime to $d$ which is linearly disjoint from $F/k$.
Then for all $n\in\NN$,
\[ \Res_{l/k}:(\Br X)[d^n]/(\Br_1 X)[d^n]\to (\Br X_l)[d^n]/(\Br_1 X_l)[d^n]\]
is an isomorphism and consequently
\[ \Res_{l/k}:(\Br X/\Br_1 X)\{d\}\to (\Br X_l/\Br_1 X_l)\{d\}\]
is an isomorphism.
\end{cor}

\begin{proof}
Let $\phi:\Gamma_k\to \Aut(\Br\overline{X}[d])$ be the map giving the Galois action. Since $l/k$ is linearly disjoint from $F/k$, we have $\Gamma_k=\Gamma_l\Gamma_F$ and hence \[\phi(\Gamma_k)=\phi(\Gamma_l\Gamma_F)=\phi(\Gamma_l),\] since $\Gamma_F$ acts trivially on $\Br\overline{X}[d]$ by assumption.

 Now by Lemma~\ref{lem:ppower}, the image of $\Gamma_l$ in $\Aut(\Br\overline{X}[d^n])$ is the same as that of $\Gamma_k$, for all $n\in\NN$. Now the first statement follows from Proposition~\ref{prop:nogrowthabvar}. Since $X$ is smooth, $\Br X$ is a torsion abelian group by \cite{Grothendieck}. Hence, the natural inclusion
\(\Br X\{d\}/\Br_1 X\{d\}\hookrightarrow (\Br X/\Br_1 X)\{d\}\)
is an equality, and similarly for $X_l$. Since
\[\frac{\Br X\{d\}}{\Br_1 X\{d\}}=  \frac{\varinjlim\Br X[d^n]}{\varinjlim\Br_1 X[d^n]}= \varinjlim \frac{\Br X[d^n]}{\Br_1 X[d^n]},\]
the second statement follows from the first one.
\end{proof}

Next, we give a uniform result which applies to the whole family ($\scrA_{k,g}$ or $\scrK_{k,g}$) and every $n\in\NN$. We will need a couple of auxiliary lemmas.

\begin{defn}\label{def:Q}
Let $d,m\in\Z_{>0}$ and let $d$ have prime factorisation $d=\prod_{i=1}^r p_i^{\epsilon_i}$. We define
$Q_{d,m}:=\{q\in\mathbb{Z} \textrm{ prime} : q\mid \prod_{i=1}^r p_i(p_i^m-1)\}$ and $N_{d,m}:=\prod_{q\in Q_{d,m}}q$.
\end{defn}

\begin{lemma}\label{lem:same action}
Let $M$ be a finite group such that $\Aut M$ has exponent dividing $e$. Let $l/k$ be a finite extension of degree coprime to $e$. Let $\phi:\Gamma_k\to\Aut M$ be a homomorphism making $M$ into a $\Gamma_k$-module. Then $\phi(\Gamma_k)=\phi(\Gamma_l)$.
\end{lemma}

\begin{proof}
We have $\phi(\Gamma_l) \leq \phi(\Gamma_k) \leq \Aut M$ with $[\phi(\Gamma_k):\phi(\Gamma_l)]$ dividing $[\Gamma_k:\Gamma_l]=[l:k]$ and with $\gcd([l:k], e) =\gcd([l:k], \exp(\Aut M)) =\gcd([l:k], \#\Aut M)=1$.
Since $M$ is finite, by Lagrange's theorem we have
\[\#\Aut M = [\Aut M : \phi(\Gamma_l)] \cdot \# \phi(\Gamma_l)= [\Aut M : \phi(\Gamma_k)] \cdot \#\phi(\Gamma_k),\]
implying that  
\([\Aut M : \phi(\Gamma_l)] =   [\Aut M : \phi(\Gamma_k)] [\phi(\Gamma_k): \phi(\Gamma_l)].\)
Hence, $[\phi(\Gamma_k): \phi(\Gamma_l)]$ divides both  $\# \Aut M $ and $[l:k]$.
Since $\gcd([l:k], \#\Aut M)=1$, it follows that $[\phi(\Gamma_k): \phi(\Gamma_l)] = 1$ and thus that $\phi(\Gamma_k) = \phi(\Gamma_l)$, as required.
\end{proof}

\begin{lemma}\label{lem:lgood}Let $d\in\Z_{>0}$ and let $g\in\Z_{>1}$. Let $k$ be a field of characteristic 0.
Let $l/k$ be a finite extension of degree coprime to $N_{d,g(2g-1)-1}$. Then for all $X\in\scrA_{k,g}\cup \scrK_{k,g}$, and all $n\in\N$, the image of $\Gamma_l$ in $\Aut(\Br\overline{X}[d^n])$ is the same as the image of $\Gamma_k$.
\end{lemma}

\begin{proof}
The proof of Lemma~\ref{lem:ppower} shows that, as an abelian group with no Galois action, $\Br\overline{X}[d^n]\cong (\Z/d^n\Z)^{g(2g-1)-\rho}$, where $1\leq \rho\leq g^2$. 
By \cite[Theorem 4.1]{HillarRhea}, the primes dividing the exponent of $\Aut ((\mathbb{Z}/d^n)^m)$ belong to $Q_{d,m}$.
 Now apply Lemma~\ref{lem:same action} with $e$ a suitable power of $N_{d,g(2g-1)-1}$ to see that the image of $\Gamma_k$ in $\Aut(\Br\overline{X}[d^n])$ is the same as the image of $\Gamma_l$.
\end{proof}

\begin{cor}\label{cor:kum-ext3}Let $d\in\Z_{>0}$ and let $g\in\Z_{>1}$. Let $k$ be a field of characteristic 0.
Then for all  $X\in\scrA_{k,g}\cup \scrK_{k,g}$, all $n \in \NN$, and all finite extensions $l/k$ of degree coprime to $N_{d,g(2g-1)-1}$,
\[ \Res_{l/k}:(\Br X)[d^n]/(\Br_1 X)[d^n]\to (\Br X_l)[d^n]/(\Br_1 X_l)[d^n]\]
is an isomorphism and
consequently \[  \Res_{l/k}:(\Br X/\Br_1 X)\{d\} \to (\Br X_l/\Br_1 X_l)\{d\}\]
is an isomorphism.
\end{cor}

\begin{proof}
Clearly, if $[l:k]$ is coprime to $N_{d,g(2g-1)-1}$, then it is coprime to $d$.
Now Lemma~\ref{lem:lgood} allows us to apply Proposition~\ref{prop:nogrowthabvar}.
\end{proof}

\section{Controlling the  Brauer group of geometrically Kummer varieties over field extensions}\label{sec:ControlKummer}

Let $k$ be a field of characteristic $0$. In this section, we give sufficient criteria for the $d$-primary part of the algebraic part of the Brauer group of a Kummer variety $X$ to be unchanged under an extension of the base field. We then combine this with the results of Section~\ref{sec:ControlAbVar} to give sufficient criteria for the $d$-primary part of the quotient of the Brauer group of a Kummer variety $X$ by the subgroup $\Br_0X$ of constant Brauer elements to be unchanged under an extension of the base field.

 The following well-known lemma tells us how to control the algebraic part of the Brauer group over finite extensions of the base field.
 
 \begin{lemma}\label{lemBr0} Let $X$ be a smooth, proper, geometrically integral variety over a number field $k$ with $\Pic \Xbar$ free of rank $r$. Let $k'/k$ be a finite Galois extension such that $\Pic X_{k'}=\Pic \Xbar $.  If a finite extension $l/k$ is linearly disjoint 
 from $k'/k$, then 
the natural homomorphism \[\Res_{l/k}: \Br_1 X/\Br_0 X \to \Br_1 X_{l}/\Br_0 X_{l}\]
 is an isomorphism. 
 \end{lemma}

 \begin{proof} We have $ \mathrm{H}^1_{\et}(k, \Pic \Xbar)  =  \mathrm{H}^1_{\textrm{cts}}(\Gal(k'/k), \Pic X_{k'})$ and, writing $l':=k' l$ for the compositum field,  $  \mathrm{H}^1_{\textrm{cts}}(\Gal(l'/l), \Pic X_{l'}) =  \mathrm{H}^1_{\et}(l, \Pic \Xbar)$.  Moreover, the restriction map induces an isomorphism
\[
\Res_{l/k}:  \mathrm{H}^1_{\textrm{cts}}(\Gal(k'/k), \Pic X_{k'}) \xrightarrow{\sim}\mathrm{H}^1_{\textrm{cts}}(\Gal(l'/l),\Pic X_{l'}).
 \]
The Hochschild-Serre spectral sequence 
\[\mathrm{H}^p_{\textrm{cts}}(\Gal(\kbar/k), \mathrm{H}_{\et}^q (\Xbar, \G_m)) \Rightarrow \mathrm{H}_{\et}^{p+q}(X, \G_m)\]
 yields $\Br_1 X/\Br_0 X = \mathrm{H}^1_{\et}(k, \Pic \Xbar)$, and similarly $\Br_1 X_l/\Br_0 X_l = \mathrm{H}^1_{\et}(l, \Pic \Xbar)$, whence the result. 
 \end{proof}

 We also have a  uniform version of  Lemma \ref{lemBr0}.

  \begin{lemma}\label{lemBr0U} Let $\frakF_{k,r}$ be the family of all smooth, proper, geometrically integral varieties over a number field $k$ with $\Pic \Xbar$ free of rank $r$.  If a finite extension $l/k$ has degree coprime to $\# \GL_r(\F_3)$, then 
the natural homomorphism \[\Res_{l/k}: \Br_1 X/\Br_0 X \to \Br_1 X_{l}/\Br_0 X_{l}\]
 is an isomorphism for all $X \in \frakF_{k,r}$. 
 \end{lemma}
 
 \begin{proof} Let $X \in \frakF_{k,r}$. Since $\Gamma_k$ acts on $\Pic \Xbar$, we have a map $\phi_X: \Gamma_k \to \Aut (\Pic \Xbar)$. Let $k'_X$ be the fixed field of $\ker \phi_X$. Then $[k'_X:k] = \# \im \phi_X$. Arguing in the same way as in the proof of \cite[Prop. 2.1]{VAV16}, we see that $\im \phi_X$ is isomorphic to a finite subgroup of $\GL_r(\Z)$ and hence to a finite subgroup of $\GL_r(\F_3)$. Hence, $[k'_X:k]$ divides $\# \GL_r(\F_3)$. Since $l/k$ is coprime to $\# \GL_r(\F_3)$, it follows that $[l:k]$ is coprime to $[k'_X:k]$ and hence that $l/k$ and $k'_X/k$ are linearly disjoint. Therefore, we can apply Lemma \ref{lemBr0} to deduce that \(\Res_{l/k}: \Br_1 X/\Br_0 X \to \Br_1 X_{l}/\Br_0 X_{l}\) is an isomorphism.
 \end{proof}
 
  \begin{remark} \label{rem:kumfreerank}
 If $X$ is geometrically a  Kummer variety over $k$ of dimension $g$, then $\Pic \Xbar$ is a finitely-generated free $\Z$-module of rank $\leq 2^{2g} + g^2$ (see \cite[Cor. 2.4]{SZ-KumVar}).
 \end{remark}

 \begin{theorem}\label{thmisoNU}
  Let $k$ be a number field,  let $d,n\in\Z_{>0}$, let $g\in\Z_{>1}$ and let $X\in   \scrK_{k,g}$.
 Fix a finite Galois extension $k'/k$ such that $\Pic X_{k'}=\Pic \Xbar $. Let $F/k$ be a finite extension such that $\Gamma_F$ acts trivially on $\Br\overline{X}[d]$.
 Let $l/k$ be a finite extension of degree coprime to $d$ which is linearly disjoint from $Fk'/k$. Then 
\[ \Res_{l/k}:(\Br X)[d^n]/(\Br_0 X)[d^n]\rightarrow (\Br X_l)[d^n]/(\Br_0 X_l)[d^n]\]
 is an isomorphism and consequently
 \[\Res_{l/k}:(\Br X/\Br_0 X)\{d\}\to(\Br X_l/\Br_0 X_l)\{d\}\]
 is an isomorphism.
\end{theorem}

\begin{proof} We have the following commutative diagram with exact rows
 \[ \begin{tikzpicture}
  \matrix (m) [matrix of math nodes, row sep=1.5em,
    column sep=3em]{
0 &\frac{(\Br_1 X)[d^n]}{(\Br_0 X)[d^n]}      & \frac{(\Br X)[d^n]}{(\Br_0 X)[d^n]}      &  \frac{(\Br X)[d^n]}{(\Br_1 X)[d^n]}     & 0\\
0 &\frac{(\Br_1 X_{l})[d^n]}{(\Br_0 X_{l})[d^n]}     & \frac{(\Br X_{l})[d^n]}{(\Br_0 X_{l})[d^n]}       &  \frac{(\Br X_{l})[d^n]}{(\Br_1 X_{l})[d^n]}      & 0,\\
};
  \path[-stealth]
  (m-1-1) edge (m-1-2) 
   (m-1-2) edge (m-1-3)
   (m-1-3) edge (m-1-4)
   (m-1-4) edge (m-1-5) 
     (m-2-1) edge (m-2-2) 
   (m-2-2) edge (m-2-3)
   (m-2-3) edge (m-2-4)
   (m-2-4) edge (m-2-5) 
  (m-1-2) edge node [left] {\tiny$\Res_{l/k}$}  node [above, sloped] {$\sim$}(m-2-2) 
   (m-1-3) edge node [left] {\tiny$\Res_{l/k}$}(m-2-3) 
   (m-1-4) edge node [left] {\tiny$\Res_{l/k}$} node [above, sloped] {$\sim$}(m-2-4) 
;
\end{tikzpicture}\]
where the  first and third vertical arrows are isomorphisms  by Lemma \ref{lemBr0} and by  Corollary~\ref{kum-ext3NU}, respectively. It follows from the commutativity of the diagram that the middle vertical arrow is also an isomorphism.
\end{proof}

We also have the following uniform version.

\begin{theorem} \label{thmiso}Let $k$ be a number field. Fix $d \in \Z_{>0}$ and $g \in \Z_{>1}$. 
Then, for all  $X \in   \scrK_{k,g}$, all $n\in\NN$, and all finite extensions $l/k$ of degree coprime to $\# \GL_{2^{2g} + g^2}(\F_3)\cdot N_{d,g(2g-1)-1}$,
the natural homomorphism  
  \[ \Res_{l/k}:(\Br X)[d^n]/(\Br_0 X)[d^n]\rightarrow (\Br X_l)[d^n]/(\Br_0 X_l)[d^n]\]
 is an isomorphism and consequently 
\[\Res_{l/k}: (\Br X/\Br_0 X)\{d\} \to (\Br X_{l}/\Br_0 X_{l})\{d\}\]
 is an isomorphism. 
\end{theorem}

\begin{proof} This follows from Lemma \ref{lemBr0U} together with Remark  \ref{rem:kumfreerank} and Corollary~\ref{cor:kum-ext3}.
\end{proof}


\section{From rational points to 0-cycles for geometrically Kummer varieties}\label{sec:0-cyc}

In this section, $k$ will always be a number field.

Recall the construction of the Brauer-Manin set for 0-cycles.
Let $X$ be a smooth, geometrically integral variety over $k$. Let $Z_0(X_l)$ denote the set of 0-cycles of $X$ over a field extension $l/k$, that is, the set of formal sums $\sum_{x \in X^{(0)}_l} n_x x$. If $\delta \in \Z_{\geq 0}$, we denote by $Z_0^\delta(X_l) \subset Z_0(X_l)$ the subset of 0-cycles of degree $\delta$, that is, 0-cycles $z:=  \sum_{x \in X^{(0)}_l} n_x x$ such that $\deg(z):= \sum_{x \in X^{(0)}_l} [l(x):l] n_x = \delta$. We can extend the Brauer-Manin pairing to 0-cycles of degree $\delta$ by defining, for any subset $B \subset \Br X$, 
\[ 
\begin{array}{lll}
\langle \ \ , \ \  \rangle_{\textrm{BM}}: &\prod_{v \in \Omega_k } Z^\delta_0(X_{k_v}) \times B & \to \Q/\Z\\
& \left( \left( \sum_{x_v \in X^{(0)}_{k_v}} n_{x_v} x_v \right)_{v },  \alpha \right) & \mapsto  \sum_{v } \inv_v\left(  \sum_{x_v \in X^{(0)}_{k_v}} n_{x_v} \Cor_{k_v(x_{v})/k_v}\alpha(x_v) \right).
\end{array} \]

\begin{defn} Let $X$ be a smooth, geometrically integral variety over a number field $k$. Let $B \subset \Br X$ be a subset of the Brauer group. We define the \defi{Brauer-Manin set associated to $B$ for 0-cycles of degree $\delta$}, denoted by $Z^{\delta}_0(X_{\A_k})^{B}$, to be the left kernel of the Brauer-Manin pairing above.
\end{defn}

\begin{defn} \label{defn:BrHP} Let $\{X_\omega\}_\omega$ be a family of smooth, geometrically integral varieties over $k$. For each $X_\omega$, let $B_\omega \subset \Br X_\omega$ be a subset of the Brauer group. We say that \defi{the $\{B_\omega\}_\omega$-obstruction to the Hasse principle for 0-cycles of degree $\delta$ is the only one for the family $\{X_\omega\}_\omega$} if $Z^{\delta}_0(X_{\omega, \A_k})^{B_\omega} \neq \emptyset$ implies $Z_0^\delta(X_\omega) \neq \emptyset$, for all $X_\omega$. 
\end{defn}

The following definition is taken from \cite{Liang} and slightly differs from the definition given in e.g. \cite{CTSD94}.
\begin{defn}[({\cite{Liang}})] \label{defn:BrWA} Let $\{X_\omega\}_\omega$ be a family of smooth, geometrically integral varieties over $k$. For each $X_\omega$, let $B_\omega \subset \Br X_\omega$ be a subset of the Brauer group.  We say that \defi{the $\{B_\omega\}_\omega$-obstruction to weak approximation for 0-cycles of degree $\delta$ is the only one}  if for any $n \in \Z_{>0}$, for any finite subset $S \subset \Omega_k$, and for any $(z_v)_{v \in \Omega_k} \in Z_0^\delta(X_{\omega, \A_k})^{B_\omega}$, there exists some $z_{n,S} \in Z_0^\delta(X_\omega)$ such that $z_{n,S}$ and $z_v$ have the same image in $\CH_0(X_{\omega, k_v})/n$ for all $v \in S$, for all $X_\omega$, where $\CH_0$ denotes the usual Chow group of 0-cycles.  
\end{defn}

\begin{remark} In Definitions \ref{defn:BrHP} and  \ref{defn:BrWA}, when for all $X_\omega$ we take e.g. $B_\omega := \Br X_\omega$ or $B_\omega := \Br X_\omega \{d\}$ for some $d \in \Z_{>0}$, we say, respectively, that the Brauer-Manin obstruction or the $d$-primary Brauer-Manin obstruction is the only one for $0$-cycles of degree $\delta$ for the family $\{ X_\omega\}_\omega$.
\end{remark}

The main conjecture governing the arithmetic properties of 0-cycles of degree 1 is the following:

\begin{conj}[{Colliot-Th\'el\`ene}] Let $\{X_\omega\}_\omega$ be the family of \emph{all} smooth, proper, geometrically integral varieties over $k$. Then the Brauer-Manin obstruction to weak approximation for 0-cycles of degree $1$ is the only one for  $\{X_\omega\}_\omega$.
\end{conj}

The main theorems of this section are the  following results analogous to \cite[Thm 3.2.1]{Liang}, relating the Brauer-Manin obstruction for rational points to that for 0-cycles.  They are essentially two versions of the same result $-$ the first version requires explicit knowledge of $X$, while the second one works uniformly for all $X \in \scrK_{k,g}$.

\begin{thm}  \label{cycKumNU}  Fix $\delta \in \Z$ and let $d \in \Z_{>0}$ be coprime to $\delta$. Let $X \in \scrK_{k,g}$. Let $F/k$ be a finite extension such that $\Gamma_F$ acts trivially on $\Br\overline{X}[d]$. Fix a finite Galois extension $k'/k$ such that $\Pic X_{k'}=\Pic \Xbar $. Suppose that, for all 
finite extensions $l/k$ of degree coprime to $d$ which are linearly disjoint from $Fk'/k$, the $d$-primary Brauer-Manin obstruction to the Hasse principle (respectively, to weak approximation) is the only one for rational points on $X_l$.
 Then the $d$-primary Brauer-Manin obstruction to the Hasse principle (respectively, to weak approximation) is the only one for 0-cycles of degree $\delta$ on $X$. 
\end{thm}

\begin{proof} We prove the theorem for weak approximation; for the Hasse principle, the argument is similar. Our proof is closely based on that of \cite[Thm 3.2.1]{Liang}. For this reason, we give just a sketch of the proof; the details can be filled in by referring to Liang's proof.

\begin{enumerate}
\item By considering the projection $\pr: X \times \PP^1_k \to X$ with section $s: x \mapsto (x,u_0)$, where $u_0 \in \PP^1_k(k)$ is some fixed rational point,  it suffices to prove the following statement:

 {\it The $d$-primary Brauer-Manin obstruction is the only one for weak approximation for 0-cycles of degree $\delta$ on $X \times \PP^1_k$. }

    Let $(z_v)_v \in Z_0^\delta((X\times \PP^1_k)_{\A_k})$ be orthogonal to $\Br(X \times \PP^1_k) \{d\}$. Let $\{ b_i\}_{i=1}^s$ be a complete set of representatives for $  (\Br (X\times \PP^1_k)/\Br_0(X\times \PP^1_k))\{ d\}$, which is finite by \cite[Cor. 2.8]{SZ-KumVar}.  Let $m:=  [k':k] \cdot \# (\Br X/ \Br_0 X) \{d\}$.   We fix $n \in \Z_{>0}$ and a finite subset $S \subset \Omega_k$.  We also fix a closed point $P:=(x_0, u_0) \in X \times \PP_k^1$ and let $\delta_P :=[k(P):k]$.

\item Let $S_0 \subset \Omega_k$ be a finite subset containing $S$ and all the archimedean places of $k$. By enlarging $S_0$ if necessary, we can assume that, for any $v \not\in S_0$, 
\begin{itemize}
\item $\langle z'_v, b_i \rangle_{\textrm{BM}} = 0$ for any $i \in \{ 1, ..., s\}$ and for any $z'_v \in Z_0(X_v \times \PP^1_v)$, and
\item $X_v$ has a $k_v$-point $x_v$ (this follows from the Lang-Weil estimates and Hensel's lifting).
\end{itemize}
\item For each $v \in S_0$, we write $z_v = z_v^+- z_v^-$, where $z_v^+$ and $z_v^-$ are effective 0-cycles with disjoint supports, and we let 
\(z_v^1:= z_v +  mn d \delta_P z_v^-.\)
Note that
\(\deg(z_v^1) = \delta +  m n d \delta_P \deg(z_v^-).\) By adding to each $z_v^1$ a suitable multiple of the 0-cycle $m n d  P_v$, where $P_v:=P \otimes_k k_v$, we obtain 0-cycles $z_v^2$ of the same degree $\Delta > 0$ for all $v \in S_0$. By construction, $\Delta \equiv \delta  \mod mnd\delta_P$. In particular, since $\gcd(d, \delta) = 1$, we have  $\gcd(\Delta, d) = 1$. Using the natural projection $\pi: X \times \PP_k^1 \to \PP_k^1$ and a moving lemma by Liang (\cite[Lemma 1.3.1]{Liang}), for each $v \in S_0$ we can find an effective 0-cycle $z_v^3$ of degree $\Delta$ such that $\pi_\ast(z_v^3)$ is separable and $z_v^3$ is sufficiently close to $z_v^2$ (and hence to $z_v^1$ and to $z_v$).  

\item  By \cite[Prop. 3.3.3]{Liang} and the discussion in the proof of  \cite[Prop. 3.4.1]{Liang}, we can find a closed point $\theta \in \PP^1_k$ sufficiently close, for $v \in S_0$, to $\pi_\ast(z_v^3)$  such that $[k(\theta):k] = \Delta$ and  $k(\theta)/k$ is linearly disjoint from $Fk'/k$, and we also obtain an adelic point $(\mathcal{M}_w)_w \in X(\A_{k(\theta)})^{\Br X \{d\}}$.

\item Since $\gcd(\Delta,d)=1$ and $k(\theta)/k$ is linearly disjoint from $Fk'/k$, Theorem~\ref{thmisoNU} shows that $(\Br X/\Br_0 X) \{d\} = (\Br X_{k(\theta)}/\Br_0 X_{k(\theta)}) \{d\}$. It follows that $(\mathcal{M}_w)_w \in X(\A_{k(\theta)})^{\Br X \{d\}}  =  X(\A_{k(\theta)})^{\Br X_{k(\theta)} \{d\}}$.  
\item By hypothesis (and by the construction of $k(\theta)/k$) there exists a global $k(\theta)$-point $\mathcal{M}$ on $X_{k(\theta)}$. By construction, $\mathcal{M}$ and $z_v$ have the same image in $\CH_0(X_v)/n$ for all $v \in S$.  
\item When viewed as a 0-cycle on $X$, $\mathcal{M}$ has degree $\Delta$. Since  $\Delta \equiv \delta \mod mnd\delta_P$, adding a suitable multiple of the degree $\delta_P$ closed point $x_0 = \pr(P)$ to $\mathcal{M}$  yields a global 0-cycle of degree $\delta$ on $X$ with the same image as $z_v$ in $\CH_0(X_{k_v})/n$ for all $v \in S$, as required. \qedhere
\end{enumerate}
\end{proof}

\begin{thm}  \label{cycKum}   Fix $d \in \Z_{>0}$ and $g \in  \Z_{>1}$.  Fix $\delta \in \Z$ coprime to $\# \GL_{2^{2g} + g^2}(\F_3)\cdot N_{d,g(2g-1)-1}$. Then for all $X \in \scrK_{k,g}$, the following holds. If for all finite extensions $l/k$ with degree coprime to $\# \GL_{2^{2g} + g^2}(\F_3)\cdot N_{d,g(2g-1)-1} $  the $d$-primary Brauer-Manin obstruction to the Hasse principle (respectively, to weak approximation) is the only one  for rational points on $X_l$,  then the $d$-primary Brauer-Manin obstruction to the Hasse principle (respectively, to weak approximation) is the only one  for 0-cycles of degree $\delta$ on $X$.
\end{thm}

\begin{proof} Again, our proof follows closely that of \cite[Thm 3.2.1]{Liang}, with some minor modifications. For this reason, we  just  sketch  the proof for weak approximation (for the Hasse principle, the argument is similar); the details can be filled in by referring to \cite{Liang}.

\begin{enumerate}
\item By considering the projection $\pr : X \times \PP^1_k \to X$ with section $s: x \mapsto (x,u_0)$, where $u_0 \in \PP^1_k(k)$ is some fixed rational point, it suffices to prove the following statement:

 {\it The $d$-primary Brauer-Manin obstruction is the only one for weak approximation for 0-cycles of degree $\delta$ on $X \times \PP^1_k$. }
  
    Let $(z_v)_v \in Z_0^\delta((X\times \PP^1_k)_{\A_k})$ be orthogonal to $\Br(X \times \PP^1_k) \{d\}$. Let $\{ b_i\}_{i=1}^s$ be a complete set of representatives for $  (\Br (X\times \PP^1_k)/\Br_0(X\times \PP^1_k))\{ d\}$, which is finite by \cite[Cor. 2.8]{SZ-KumVar}.  Let $m:=   \# \GL_{2^{2g} + g^2}(\F_3)  \cdot \# (\Br X/ \Br_0 X) \{d\}$.  We fix $n \in \Z_{>0}$ and a finite subset $S \subset \Omega_k$.  We also fix a closed point $P:=(x_0, u_0) \in X \times \PP_k^1$ and let $\delta_P :=[k(P):k]$.

\item Let $S_0 \subset \Omega_k$ be a finite subset containing $S$ and all the archimedean places of $k$. By enlarging $S_0$ if necessary, we can assume that, for any $v \not\in S_0$, 
\begin{itemize}
\item $\langle z'_v, b_i \rangle_{\textrm{BM}} = 0$ for any $i \in \{ 1, ..., s\}$ and for any $z'_v \in Z_0(X_v \times \PP^1_v)$, and
\item $X_v$ has a $k_v$-point $x_v$ (this follows from the Lang-Weil estimates and Hensel's lifting).
\end{itemize}

\item For each $v \in S_0$, we write $z_v = z_v^+- z_v^-$, where $z_v^+$ and $z_v^-$ are effective 0-cycles with disjoint supports, and we let 
\(z_v^1:= z_v +  mn d \delta_P z_v^-.\)
Note that
\(\deg(z_v^1) = \delta +  m n d \delta_P \deg(z_v^-).\) By adding to each $z_v^1$ a suitable multiple of the 0-cycle $m n   P_v$, where $P_v:=P \otimes_k k_v$, we obtain 0-cycles $z_v^2$ of the same degree $\Delta > 0$ for all $v \in S_0$, where we can take $\Delta$ to satisfy $\Delta \equiv \delta  \mod mn \delta_P  N_{d,g(2g-1)-1}$. In particular, since $\delta$ is coprime to $\# \GL_{2^{2g} + g^2}(\F_3)\cdot N_{d,g(2g-1)-1}$, we deduce that $\Delta$ is coprime to $\# \GL_{2^{2g} + g^2}(\F_3) \cdot  N_{d,g(2g-1)-1}$. Using the natural projection $\pi: X \times \PP_k^1 \to \PP_k^1$ and a moving lemma by Liang (\cite[Lemma 1.3.1]{Liang}), for each $v \in S_0$ we can find an effective 0-cycle $z_v^3$ of degree $\Delta$ such that $\pi_\ast(z_v^3)$ is separable and $z_v^3$ is sufficiently close to $z_v^2$ (and hence to $z_v^1$  and to $z_v$).

\item Arguing as in \cite[Proof of Thm 3.2.1]{Liang}, we obtain a  field extension $k(\theta)/k$ such that $[k(\theta):k] = \Delta$ and an adelic point $(\mathcal{M}_w)_w \in X(\A_{k(\theta)})^{\Br X \{d\}}$.
\item By Theorem \ref{thmiso}, we have $(\Br X/\Br_0 X) \{d\} = (\Br X_{k(\theta)}/\Br_0 X_{k(\theta)}) \{d\}$. Hence,  $(\mathcal{M}_w)_w \in X(\A_{k(\theta)})^{\Br X \{d\}}  =  X(\A_{k(\theta)})^{\Br X_{k(\theta)} \{d\}}$.  
\item By hypothesis (and by the construction of $k(\theta)/k$) there exists a global $k(\theta)$-point $\mathcal{M}$ on $X_{k(\theta)}$. By construction, $\mathcal{M}$ and $z_v$ have the same image in $\CH_0(X_v)/n$ for all $v \in S$.  
\item When viewed as a 0-cycle on $X$, $\mathcal{M}$ has degree $\Delta$. Since  $\Delta \equiv \delta \mod mnd\delta_P$, adding a suitable multiple of the degree $\delta_P$ closed point $x_0 = \pr(P)$ to $\mathcal{M}$  yields a global 0-cycle of degree $\delta$ on $X$ with the same image as $z_v$ in $\CH_0(X_{k_v})/n$ for all $v \in S$, as required. \qedhere
\end{enumerate}
\end{proof}

\begin{remark} \label{rem2} Let $X$ be a Kummer variety over $k$.  By \cite[Thm 1.7]{CreutzViray-DegBMO}, it follows that $X_l(\A_l)^{\Br X_l} \neq \emptyset$ if and only if $X_l(\A_l)^{\Br X_l\{2\}} \neq \emptyset$, for any finite extension $l/k$.  Hence, for any such $X$,  saying  ``the 2-primary Brauer-Manin obstruction to the Hasse principle is the only one for rational points on $X_l$'' is equivalent to saying ``the Brauer-Manin obstruction to the Hasse principle is the only one for rational points on $X_l$''. 
\end{remark}

\begin{cor}\label{cor:Kummer} 
Let $X$ be a Kummer variety over $k$ and let $\delta$ be an odd integer. Suppose that the Brauer-Manin obstruction to the Hasse principle is the only one for rational points on $X_l$ for all finite extensions $l/k$.
 Then the $2$-primary Brauer-Manin obstruction to the Hasse principle is the only one for $0$-cycles of degree $\delta$ on $X$. 
\end{cor}

\begin{proof}
This is an immediate consequence of \cite[Thm 1.7]{CreutzViray-DegBMO} and the proof of Theorem~\ref{cycKumNU} applied with $d=2$. 
\end{proof}

\begin{conj} [Skorobogatov] \label{conjsko} The Brauer-Manin obstruction is the only obstruction to the Hasse
principle and weak approximation on K3 surfaces over number fields.
\end{conj}

By taking into account Remark \ref{rem2} and Conjecture  \ref{conjsko}, and by running the proof for  $d=2$, we obtain the following corollary to Theorem \ref{cycKumNU}.

\begin{cor} \label{corsimple}Let $X$ be a Kummer surface over $k$ and let $\delta$ be an odd integer. Then, conditionally on Conjecture \ref{conjsko}, the 2-primary Brauer-Manin obstruction to the Hasse principle is the only one  for 0-cycles of degree $\delta$ on $X$.  
\end{cor}

\section{Transferring the emptiness of the Brauer-Manin set over field extensions
}\label{sec:transfer}
We establish the following result about transferring sufficiency of the Brauer-Manin obstruction over some field extensions (similar ideas have appeared in  \cite{riman}). 

\begin{thm}  Let $k$ be a number field and let $X$ be a Kummer variety over $k$. Suppose that $(\Br X/\Br_0 X)\{2\} \cong\Z/2\Z$ and $X(\A_k) \neq \emptyset$. Then there exist uncountably many finite extensions $l/k$  such that 
\[ X(\A_k)^{\Br X} = \emptyset \Longrightarrow  X(\A_l)^{\Br X_l} = \emptyset.\]
In particular, if the Brauer-Manin obstruction to the Hasse principle  is the only one for $k$-rational points, then it is also the only one  for $l$-rational points, for all field extensions $l/k$ as above. 
\end{thm}
\begin{proof} Let $\alpha \in \Br X \{2\}$ be such that its image in $\Br X/ \Br_0 X$ generates $(\Br X/ \Br_0 X)\{2\}$. Let $\calX$ be an integral model for $X$ over $\Spec \calO_{S}$, where $S \subset \Omega_k$ is some finite subset containing the archimedean places and such that $\alpha$ comes from an element in $\Br\calX\{2\}$. 
Since $\Br \calO_v = 0$ for all $v \not \in S$, we have $\langle x_v, \alpha \rangle_{\textrm{BM}} = 0$ for any $x_v \in X(k_v)$ and $v \not\in S$. 

Let $l/k$ be a finite extension of odd degree which is completely split at all places in $S$. Suppose that  $ X(\A_l)^{\Br X_l} \neq \emptyset$. 
It follows that $ X(\A_l)^{\alpha_l} \neq \emptyset$, where $\alpha_l:=\Res_{l/k}\alpha$.
Let $(x_w) \in X(\A_l)^{\alpha_l}$. Since $\inv_w \alpha_l(x_w) = 0$ for any place $w$ of $l$ lying above a place $v$ of $k$ with $v\notin S$, we have 
\[0= \sum_{w \in \Omega_l } \inv_w(\alpha_l(x_w)) =  \sum_{v \in S} \sum_{w|v} \inv_w(\alpha_l(x_w)).\]
Since $l/k$ is completely split at all places in $S$, for each place $w|v$ with $v \in S$ we have \mbox{$[l_w:k_v]=1$} and hence there exists some $u^{(w)}_v \in X(k_v)$ with $\inv_w(\alpha_l(x_w)) = \inv_v(\alpha(u_v^{(w)}))$.
It follows that
\begin{equation}\label{Av} \sum_{v \in S} \sum_{w|v} \inv_w(\alpha_l(x_w)) = \sum_{v \in S} \sum_{w|v}  \inv_v(\alpha(u^{(w)}_v)) = 0.\end{equation}
We now construct an adelic point $(x'_v) \in X(\A_k)^{\Br X\{2\}}$. By assumption, $X(\A_k) \neq \emptyset$, so let $(x''_v) \in X(\A_k)$. Let $x'_v := x''_v$ for all $v \not\in S$. For $v \in S$, we proceed as follows. Define
\[ A_v:=  \sum_{w|v}  \inv_v(\alpha(u^{(w)}_v)).\]
If $A_v = 0$, then there exists some $w | v$ such that $\inv_v \alpha(u_v^{(w)}) = 0$, since $\alpha$ has order $2$ in $\Br X/\Br_0 X$ and $\# \{ w| v\} = [l:k]$ is odd. 
In this case, let $x'_v := u_v^{(w)}$.
Note that, from \ref{Av},  there is an even number (possibly zero) of places $v \in S$ such that $A_v = 1/2$; let $\calV$ be the set of all such $v \in S$. If $\calV = \emptyset$, then $A_v = 0$ for all $v \in S$. This means that we have already constructed an adelic point $(x'_v) \in X(\A_k)$ with the property that $(x'_v) \in X(\A_k)^\alpha = X(\A_k)^{\Br X \{2\}}$.
So let us suppose that $\calV \neq \emptyset$ and pair the places in $\calV$, writing $\calV = \{v_1, v_1', v_2, v'_2, ..., v_r, v'_r\}$. For $i = 1, ..., r$, we do the following. Since $A_{v_i} = 1/2$ and  $A_{v'_i} = 1/2$, we can pick some  $u^w_{v_i} \in X(k_{v_i})$ and some  $u^{w'}_{v'_i} \in X(k_{v'_i})$ such that $\inv_{v_i} \alpha(u^w_{v_i} ) = \inv_{v'_i} \alpha(u^{w'}_{v'_i})  = 1/2$. We then set $x'_{v_i} := u^w_{v_i}$ and $x'_{v'_i}:= u^{w'}_{v'_i}$. At the end of this process, we obtain  an adelic point $(x'_v) \in X(\A_k)$ which, by construction, has the property that $(x'_v) \in X(\A_k)^\alpha = X(\A_k)^{\Br X \{2\}}$. 

Hence, for any finite extension $l/k$ of odd degree which is completely split at all places in $S$,  we have shown that $ X(\A_l)^{\Br X_l} \neq \emptyset$ implies that $ X(\A_k)^{\Br X\{2\}} \neq \emptyset$. By \cite[Thm 1.7]{CreutzViray-DegBMO}, the latter implies that $ X(\A_k)^{\Br X} \neq \emptyset$, as required. The fact that there exist uncountably many such field extensions $l/k$ follows from e.g. the proof of \cite[Thm 4.1]{MazurRubin-DioStab}. The very last statement of the theorem is clear.
\end{proof}

\renewcommand{\bibname}{References}

\bibliographystyle{alpha}
\bibliography{bibshort}

\end{document}